\documentclass[11pt]{article}

\usepackage{fullpage}
\usepackage{amsmath, amsthm, amsfonts, amssymb, amstext, mathrsfs, enumerate}
\usepackage{graphicx, ragged2e, lscape, framed, xcolor}
\usepackage{subfiles}

\theoremstyle{plain}
\newtheorem{theorem}{Theorem}[section]

\newtheorem{problem}[theorem]{Problem}
\newtheorem{claim}{Claim}[section]
\newtheorem{remark}[theorem]{Remark}

\numberwithin{equation}{section}
\allowdisplaybreaks

\newcommand{\affl}[3]{\noindent #1, Email: {\tt #2}\\ \textsc{#3}\\[1.5pt]}

\usepackage[pagebackref]{hyperref}
\hypersetup{
	colorlinks=true,
    urlcolor=purple,
	linkcolor=purple,
    citecolor=purple,
}

\DeclareMathOperator{\tr}{tr}

\DeclareMathOperator{\cone}{Cone}
\DeclareMathOperator{\sgn}{sgn}
\DeclareMathOperator{\diag}{diag}

\DeclareMathOperator{\lin}{\ell_{\infty}}
\def\1{\mbox{\boldmath $1$}}
\def\cvec{\mbox{\boldmath $c$}}
\def\w{\mbox{\boldmath $w$}}

\def\s{\mbox{\boldmath $s$}}
\def\x{\mbox{\boldmath $x$}}

\def\z{\mbox{\boldmath $z$}}
\newcommand{\ip}[2]{\langle #1, #2 \rangle}

\title{\textbf{Stability of maximal relative projection constants}}
\author{Hitesh Kumar, Bojan Mohar, Seyed Ahmad Mojallal, Shivaramakrishna Pragada}
\date{}

\begin{document}
\maketitle
\begin{abstract} For positive integers $n\ge r$, let $\lambda(r,n)$ denote the \emph{maximal relative projection constant} of $r$-dimensional subspaces of $\ell_\infty^n$ and $\lambda(r)$ denote the \emph{maximal absolute projection constant}, respectively. It is known that for any fixed $r$, $\lambda(r,n)$ is a non-decreasing sequence with limit $\lambda(r)$ as $n\to \infty$. A natural question is whether $\lambda(r,n)$ stabilizes at $\lambda(r)$ for some $n>r$. We prove that for any fixed $r$, 
\[\lambda(r,n)=\lambda(r)
\qquad\text{for every}\qquad
n\ge 2^{r}\binom{r+1}{2}.\] 
This answers a question of Basso. The technique used is of independent interest. 
\end{abstract}

\noindent
\textbf{Keywords:} Maximal relative projection constant, Maximal absolute projection constant, Projection constants, Seidel adjacency matrix, Convexity    

\noindent
\textbf{MSC2020:} 05C50, 15A42, 15A60, 46B20, 52A38

\section{Introduction}

\subsection{Background}

For $n\in \mathbb{N}$, define $[n]:=\{1, \ldots, n\}$. For an index set $\mathcal{I}\subseteq \mathbb{N}$, let
\[\lin(\mathcal{I}):=\{(x_i)_{i\in \mathcal{I}}: x_i\in \mathbb{R},\  \sup_{i\in \mathcal{I}} |x_i| < \infty \}.\]
We will denote $\lin([n])$ simply by $\ell_\infty^n$. 

Let $X$ be a normed space over $\mathbb{R}$ and $Y$ be a linear subspace of $X$. A linear onto map $P:X\to Y$ is called a \emph{projection} if $P^2 = P$. The \emph{relative projection constant} of $Y$ is defined to be 
\[ \lambda(Y, X):=\inf \{\|P\| : P \text{ is a projection from }X\to Y\},\]
where $\|P\|$ denotes the \emph{operator norm} of $P$. The \emph{absolute projection constant} of $Y$ is defined to be 
\[ \lambda(Y):=\sup \{\lambda(T(Y), X): X \text{ is a normed space and } T:Y\to X \text{ is a linear isometry}\}.\]
The question of determining these projection constants is quite old, very difficult, and related to many other concepts in functional analysis. We refer to the recent papers of Basso \cite{Basso_2019} and Kobos \cite{Kobos_2025} and the references therein for detailed background. Our focus will be the finite-dimensional case over the field $\mathbb{R}$. It is well-known that any finite-dimensional normed space $X$ can be embedded isometrically in $\lin(\mathcal{I})$ for some $\mathcal{I}\subset \mathbb{N}$. And it is known that $\lambda(Y) = \lambda(T(Y), \ell_\infty(\mathbb{N}))$ for any isometric embedding $T:Y\to \ell_\infty(\mathbb{N})$. There are two parameters of interest.  

For $r, n\in \mathbb{N}$ with $n\ge r$, define 
\begin{equation}\label{eq:defn}
   \lambda(r, n):= \sup\{\lambda(Y, \ell_\infty^n): \dim(Y) = r\} \quad \text{and}\quad  \lambda(r):= \sup\{\lambda(Y): \dim(Y) = r\}. 
\end{equation}
The quantities $\lambda(r,n)$ and $\lambda(r)$ are called \emph{maximal relative projection constant} and \emph{maximal absolute projection constant}, respectively. Clearly, $\lambda(r, n)\le \lambda(r)$ for all $n \ge r$. It is known that $\lambda(r)$ is a well-defined real number (cf. \cite{Basso_2019}). Moreover,
\begin{equation}
   \lambda(r,n)\le \lambda(r,n+1) 
\end{equation}
for all $n\ge r$ (cf. \cite{Chalmers_Lewicki_2010}), and
\begin{equation}\label{eq:sup_r_d_equal_absolute}
  \sup_{n\ge r} \lambda(r,n) = \lambda(r)  
\end{equation}
for any fixed $r$. 

Determining $\lambda(r)$ or $\lambda(r,n)$ is extremely difficult. We give a brief summary of the known values. An immediate application of the Hahn--Banach Theorem gives $\lambda(1) = 1$. In 1960, Gr\"{u}nbaum conjectured $\lambda(2)=\frac{4}{3}$, which was eventually resolved in 2010 by Chalmers and Lewicki \cite{Chalmers_Lewicki_2010}, with an alternative proof given by Basso \cite{Basso_2019} in 2019. In 2023, in a breakthrough paper, Der\c{e}gowska and Lewandowska \cite{Deregowska_Lewandowska_2023} determined the new values $\lambda(3) = \frac{1+\sqrt{5}}{2}$, $\lambda(7) = \frac{5}{2}$ and $\lambda(23) = \frac{14}{3}$. Their results show a deep connection between the projection constant $\lambda(r)$ and the existence of maximal equiangular tight frames in $\mathbb{R}^r$. As for general bounds, Kadets and Snobar \cite{Kadets_Snobar_1971} established 
\begin{equation}\label{eq:sqrt_r}
  \lambda(r)\le \sqrt{r}, 
\end{equation}
and K\"{o}nig \cite{Konig_1985} proved that the bound in \eqref{eq:sqrt_r} is asymptotically sharp. The following general inequality was first claimed by K\"{o}nig and Tomczak-Jaegermann \cite{Konig_Jaegermann_1994}, but later their proof was found to be erroneous, as shown in \cite{Chalmers_Lewicki_2009}. The correct and indeed a very short proof was given by Der\c{e}gowska and Lewandowska \cite{Deregowska_Lewandowska_2023}.  

\begin{theorem}[\cite{Deregowska_Lewandowska_2023}]\label{thm:D_L_bound} For any $r\in \mathbb{N}$, 
   \begin{equation}
       \lambda(r)\le \frac{r + \sqrt{r+2}}{1+\sqrt{r+2}}.
   \end{equation}
\end{theorem}

The equality case for the inequality in Theorem \ref{thm:D_L_bound} was characterized by Kobos \cite{Kobos_2025}. Refer to Kobos \cite{Kobos_2025} for an interesting and detailed overview of the above results. Recently, Sivashankar, Tang and Wakhare \cite{Sivashankar_Tang_Wakhare_2026} (see also \cite{Kumar_Mohar_Mojallal_Pragada_2026}) have exploited the connection between projection constants and graph eigenvalues to make significant progress on the Hong--Nikiforov problem for the $k$-th eigenvalue of graphs. They also improved the bound in Theorem \ref{thm:D_L_bound} in certain cases.  

\subsection{Stability of $\lambda(r,n)$: Problem and results}

As discussed earlier, for any fixed $r$, the sequence $\lambda(r,n)$ is a non-decreasing convergent sequence in $n$ with limit $\lambda(r)$. The known values of $\lambda(r)$ for $r\in \{1,2,3,7, 23\}$ suggest that $\lambda(r,n)$ stabilizes once $n\ge \frac{r(r+1)}{2}$. More precisely, the following is known.
\begin{enumerate}[$(i)$]
   
    \item $\lambda(1,n) = 1$ $(n\ge 1)$. Follows from the Hahn--Banach Theorem.
    \item $\lambda(2,n) = \frac{4}{3}$ $(n\ge 3)$. See Chalmers and Lewicki \cite{Chalmers_Lewicki_2010}.
    \item $\lambda(3,n) = \frac{1+\sqrt{5}}{2}$ $(n\ge 6)$. See Der\c{e}gowska and Lewandowska \cite{Deregowska_Lewandowska_2023}. 
    \item $\lambda(7,n) = \frac{5}{2}$ $(n\ge 28)$. See Der\c{e}gowska and Lewandowska \cite{Deregowska_Lewandowska_2023}.
    \item $\lambda(23, n) = \frac{14}{3}$ $(n\ge 276)$. See Der\c{e}gowska and Lewandowska \cite{Deregowska_Lewandowska_2023}.
\end{enumerate}
 
These observations suggest a natural question, namely, does the sequence $\lambda(r,n)$ always stabilize? Basso \cite{Basso_2019} claimed a proof of this, but later reported an error \cite{Basso_2024_corrigendum} (cf. \cite{Kobos_2025}). Basso \cite{Basso_2024_corrigendum} then formally posed the following equivalent question about the stability of $\lambda(r,n)$.

\begin{problem}[Basso {\cite[Question A]{Basso_2024_corrigendum}}] Is it true that for any $r$, there exists $n_r > r$ such that $\lambda(r,n) = \lambda(r)$ for all $n\ge n_r$? What is the smallest value of $n_r$?
\end{problem}

In this article, we address this problem and show the following. 

\begin{theorem}\label{thm:stability_rel_proj_constant} For $r\in \mathbb{N}$, let 
\[N_r:= 2^{r}{r+1\choose 2}.\]
Then, for any $n\ge N_r$, we have 
\[\lambda(r,n) = \lambda(r).\]
\end{theorem}

Our proof of Theorem \ref{thm:stability_rel_proj_constant} makes use of a characterization for $\lambda(r,n)$ as the sum of the top $r$ eigenvalues of certain weighted Seidel adjacency matrices of graphs of order $n$ (Theorem \ref{thm:relative_top_r_sum_characterization}). So the stability problem for $\lambda(r,n)$ is equivalent to showing that in order to determine $\lambda(r)$, it suffices to work with (weighted) graphs of order at most $N_r$. Such reductions have been previously used in spectral graph theory for graph eigenvalue extremization problems; indeed, the problem of estimating a certain function of graph eigenvalues over \emph{all} graphs is often reduced to a \emph{finite} computational problem for weighted graphs of certain finite order. Such instances include Terpai's \cite{Terpai_2011} 
(cf. \cite{Liu_2024}) maximization proof of the Nordhaus--Gaddum sum for spectral radius; Breen, Riasanovsky, Tait and Urschel's \cite{Breen_Riasanovsky_Tait_Urschel_2022} maximization 
proof of the \emph{graph spread}; and more recently Kumar, Liu, Monterde, Pragada and Tait's \cite{Kumar_Liu_Monterde_Pragada_Tait_2026}  maximization of the \emph{spectral sum} of graphs. Our proof closely follows the ideas in \cite{Kumar_Liu_Monterde_Pragada_Tait_2026}, where Carath\'eodory's Theorem is used to arrive at the finite reduction. 

It is known that for $n\ge r$ there always exists an $r$-dimensional subspace $Y$ of $\ell_\infty^n$ such that $\lambda(Y, \ell_\infty^n) = \lambda(r,n)$ (cf. \cite{Kobos_2025}), equivalently, the supremum in the definition of $\lambda(r,n)$ in \eqref{eq:defn} is actually a maximum. As a consequence of Theorem \ref{thm:stability_rel_proj_constant}, there always exists an $r$-dimensional subspace $Y$ of $\ell_\infty^n$ such that $\lambda(Y, \ell_\infty^n) = \lambda(Y, \ell_\infty(\mathbb{N})) = \lambda(r)$.

Moreover, we have also shown that the minimum value $n_r$ that guarantees stability satisfies
\[ n_r\le 2^r{r+1\choose 2}.\]
It is possible to obtain minor improvements to the above bound by a careful analysis of our proofs. But it does not seem easy to us to establish a bound for $n_r$ which is polynomial in $r$. The known values of $\lambda(r)$ mentioned above suggest that the following question may have an affirmative answer. 

\begin{problem}
Is it true that $n_r = O(r^2)$ as $r\to \infty$? 
\end{problem}

The remainder of the article is devoted to a proof of Theorem \ref{thm:stability_rel_proj_constant}.

\section{Proof of Theorem \ref{thm:stability_rel_proj_constant}}

Fix $r, n\in \mathbb{N}$ such that $n\ge r$. Let $M\in \mathbb{R}^{n\times n}$ be a symmetric matrix. Then the eigenvalues of $M$ are real, which we enumerate as 
\[\lambda_1(M)\ge \ldots\ge \lambda_n(M).\]

Let $G = (V(G), E(G))$ be a graph of order $n$ with vertex set $V(G) = \{1, \ldots, n\}$. We write $i\sim j$ if vertices $i$ and $j$ are adjacent; we write $i\nsim j$ otherwise. The \emph{Seidel adjacency matrix} of $G$ is the $n\times n$ matrix $S(G)$ with entries 
\[S(G)_{ij} = 
\begin{cases}
    0 & \text{if }i=j;\\
    -1 & \text{if }i\sim j;\\
    1 & \text{if }i\nsim j. 
\end{cases}\]
Let $\mathcal{G}(n)$ denote the set of all simple graphs of order $n$. Let $I_n$ denote the identity matrix of order $n$. Let $\w\in \mathbb{R}^n$ and $U\in \mathbb{R}^{n\times r}$. Define $D_{\w}:=\diag(\w)$. We say that the vector $\w$ is a \emph{feasible vector} if $\w$ is a unit non-negative vector. We say that $U$ is a \emph{feasible matrix} if $U^\top U = I_r$. For given feasible $U$ and $\w$, define 
\[ \sigma(U, \w):=\sum_{i,j=1}^n w_iw_j|\ip{\x_i}{\x_j}|,\]
where $\x_i^\top$ is the $i$-th row of $U$. For a feasible vector $\w$ and a graph $G\in \mathcal{G}(n)$, define
\[ M(G, \w):= D_{\w}(I_n+S(G))D_{\w}.\]

Chalmers and Lewicki \cite{Chalmers_Lewicki_2010} (see also Foucart and Skrzypek \cite{Foucart_Skryzpek_2017}) established the following connection between the relative projection constant $\lambda(r,n)$, the sum of the top $r$ eigenvalues of weighted Seidel matrices of order $n$ and the quantity $\sigma(U, \w)$ for feasible $U$ and $\w$.

\begin{theorem}[Chalmers--Lewicki \cite{Chalmers_Lewicki_2010}, cf. {\cite[Theorem 1]{Foucart_Skryzpek_2017}}]\label{thm:relative_top_r_sum_characterization} For $n\ge r$, we have 
\begin{align*}
    \lambda(r,n) & = \max \left\{\sum_{i=1}^r \lambda_i(M(G, \w)):\, \w \in \mathbb{R}_{\ge 0}^n,\, \|\w\|_2= 1,\, G\in \mathcal{G}(n)\right\}\\
    & = \max\left\{\sigma(U, \w): \, \w \in \mathbb{R}_{\ge 0}^n,\, \|\w\|_2= 1,\, U\in \mathbb{R}^{n\times r},\, U^\top U = I_r\right\},
\end{align*}
where $\x_i^\top$ is the $i$-th row of $U$.
\end{theorem}

\begin{remark} \label{remark:lambda_is_positive}
For every \(n\ge r\), we have
$\lambda(r,n)\ge \lambda(r,r)=1$.
Indeed, if \(n=r\), then every feasible matrix \(U\) is orthogonal, and hence
\[
\sigma(U,\w)
=\sum_{i,j=1}^r w_iw_j
 \left|\ip{\x_i}{\x_j}\right|
=\sum_{i=1}^r w_i^2
=1.
\]
Thus, \(\lambda(r,r)=1\). For any fixed $r$, the sequence $\lambda(r,n)$ is a non-decreasing convergent sequence in $n$, which gives $\lambda(r,n)\ge 1$.
\end{remark}

The two characterizations for $\lambda(r,n)$ in Theorem \ref{thm:relative_top_r_sum_characterization} are essentially an application of the following Ky--Fan maximum principle.

\begin{theorem}[Ky--Fan maximum principle {\cite[Corollary~4.3.39]{Horn_Johnson_2013}}]\label{thm:Ky_Fan}
Let $M\in\mathbb{R}^{n\times n}$ be a symmetric matrix. Then, for every $1\leq k\leq n$,
\[
\sum_{i=1}^k \lambda_i(M)
=
\max_{\substack{U\in\mathbb{R}^{n\times k}\\ U^{\mathsf T}U=I_k}}
\tr\left(U^{\mathsf T}MU\right).
\]
The maximum is attained if and only if the columns of $U$ span an invariant
subspace associated with the $k$ largest eigenvalues of $M$.
\end{theorem}

We will now prove Theorem \ref{thm:stability_rel_proj_constant} in a series of claims.  

We say that a feasible pair $(U,\w)$ is a \emph{maximizing pair} if $\sigma(U, \w) = \lambda(r,n)$. By Theorem \ref{thm:relative_top_r_sum_characterization}, such a maximizing pair $(U, \w)$ exists. Let $\x_i^\top$ $(1\le i\le n)$ denote the $i$-th row of this $U$. Define $B\in \mathbb{R}^{n\times n}$ to be the matrix whose $ij$-th entry is $|\ip{\x_i}{\x_j}|$. 

\begin{claim}\label{claim:spectral_radius_B}
We have $\lambda(r,n) = \lambda_1(B)$ and $\w$ is a unit $\lambda_1$-eigenvector for $B$. 
\end{claim}

\begin{proof} By the choice of $(U, \w)$ and the Ky--Fan principle for first eigenvalue, we have 
\[ \lambda(r,n) =\w^\top B\w \le \lambda_1(B).\]
Conversely, let $\z$ be a unit $\lambda_1$-eigenvector of $B$. By the Perron-Frobenius Theorem, $\z$ can be chosen to be non-negative, and so $\z$ is feasible. Thus, by Theorem \ref{thm:relative_top_r_sum_characterization},
\[\lambda(r,n) \ge \z^\top B\z = \lambda_1(B).\]
Thus, 
\[ \lambda(r, n) = \w^\top  B\w = \lambda_1(B),\]
and the claim follows.
\end{proof}

\begin{claim} There exists a graph $G\in \mathcal{G}(n)$ and a symmetric matrix $H\in \mathbb{R}^{r\times r}$ such that 
\begin{equation}\label{eq:H_equation}
    M(G, \w) U = UH.
\end{equation}
\end{claim}

\begin{proof} Consider a graph $G\in \mathcal{G}(n)$ such that 
\[ (I_n+S(G))_{ij} = 
\begin{cases}
1 & \text{ if }\ip{\x_i}{\x_j}\ge 0;\\
-1 & \text{ if }\ip{\x_i}{\x_j}< 0.
\end{cases}\]
Note that 
\begin{align*}
   \lambda(r,n) & = \sigma(U, \w)\\
   & = \tr(U^\top M(G, \w)U)\quad (\text{by the choice of }G)\\
   & \le \sum_{i=1}^r \lambda_i(M(G, \w)) \quad (\text{Ky--Fan principle})\\
   & \le \lambda(r,n)\quad (\text{by Theorem \ref{thm:relative_top_r_sum_characterization}}).   
\end{align*}
Thus, 
\[\tr(U^\top M(G, \w)U) = \sum_{i=1}^r \lambda_i(M(G, \w)).\]
Again, by the Ky--Fan principle, the columns of $U$ span an invariant subspace associated with the $r$ largest eigenvalues of $M(G, \w)$. Therefore,
\[
M(G,\w)U=UH,
\qquad
H:=U^\top M(G,\w)U.
\]
In particular, since \(M(G,\w)\) is symmetric, \(H\in\mathbb{R}^{r\times r}\)
is symmetric.
\end{proof}

\begin{claim}\label{claim:H_quadratic_form} For $i\in [n]$, 
    \[\x_i^\top H\x_i =  \lambda(r,n)w_i^2.\]
\end{claim}

\begin{proof}
Taking the $i$-th row of \eqref{eq:H_equation} gives
\[
w_i\sum_{j=1}^n w_j(I_n+S(G))_{ij}\x_j^\top
=
\x_i^\top H.
\]
Equivalently, after transposing and using the symmetry of \(H\), we obtain
\[w_i\sum_{j=1}^n w_j(I_n+S(G))_{ij}\x_j = H\x_i.\]
Taking inner product with $\x_i$ on both sides, we get
\begin{align*}
    \x_i^\top H\x_i & = w_i\sum_{j=1}^n w_j(I_n+S(G))_{ij}\ip{\x_i}{\x_j}\\
    & = w_i\sum_{j=1}^n w_j|\ip{\x_i}{\x_j}|\\
    & = w_i (\lambda_1(B)w_i) = \lambda(r,n)w_i^2 \quad (\text{by Claim \ref{claim:spectral_radius_B}}). \qedhere
\end{align*}
\end{proof}

The following claim is the key step in our proof, which allows us to reduce the order of the matrix $U$ that we need to work with to determine $\lambda(r,n)$.

\begin{claim}[Cloning]\label{claim:cloning} Let $C\subseteq [n]$ be such that the following two conditions hold:
\begin{enumerate}[$(i)$]
    \item $\ip{\x_i}{\x_j}\ge 0$ for all $i,j\in C$, and
    \item there exists $\cvec\in \mathbb{R}^n_{\ge 0}$ such that $c_i = 1$ if $i\notin C$ and
    \[\sum_{i\in C}c_i \x_i\x_i^\top = \sum_{i\in C}\x_i\x_i^\top.\]
    \end{enumerate}
    Define $\widetilde{w}$ and $\widetilde{\x}_i$ $(1\le i\le n)$ such that 
    \[ \widetilde{w}_i := \sqrt{c_i}w_i \qquad 
    \widetilde{\x}_i := \sqrt{c_i}\x_i.
    \]
Let $\widetilde{U}\in \mathbb{R}^{n\times r}$ be the matrix whose $i$-th row is $\widetilde{\x}_i$. Then $(\widetilde{U}, \widetilde{w})$ is a maximizing pair for $\lambda(r, n)$.
\end{claim}

\begin{proof} First we need to show that $(\widetilde{U}, \widetilde{w})$ is feasible. We have 
\begin{align*}
    \widetilde{U}^\top \widetilde{U} 
    &  = \sum_{i=1}^n \widetilde{\x}_i\widetilde{\x}_i^\top\\
    & = \sum_{i=1}^n c_i\x_i\x_i^\top\\
    & = \sum_{i=1}^n \x_i\x_i^\top \quad(\text{by condition $(ii)$})\\
    & = U^\top U\\
    & = I_r. 
\end{align*}
So $\widetilde{U}$ is feasible. 

Clearly, $\widetilde{w}$ is a non-negative vector. Moreover, 
\begin{align*}
    \lambda(r,n) \left(\|\widetilde{\w}\|_2^2 - \|\w\|_2^2\right)
    & = \lambda(r,n) \sum_{i\in C}(c_i - 1)w_i^2\\
    & = \sum_{i\in C} (c_i-1)\x_i^\top H\x_i\quad (\text{by Claim \ref{claim:H_quadratic_form}})\\
    & = \tr\left(H \sum_{i\in C}(c_i-1)\x_i\x_i^\top \right)\\
    & = 0 \quad (\text{by condition $(ii)$}).
\end{align*}
Since $\lambda(r,n)>0$ by Remark \ref{remark:lambda_is_positive}, we have $\|\widetilde{\w}\|_2^2 = \|\w\|_2^2 = 1$. So $\widetilde{w}$ is also feasible. 

Now, we need to argue that $(\widetilde{U}, \widetilde{w})$ is a maximizing pair. Note that
\[\widetilde{w}_i\widetilde{w}_j|\ip{\widetilde{\x}_i}{\widetilde{\x}_j}| = c_ic_jw_iw_j|\ip{\x_i}{\x_j}|.\]
Thus, 
\begin{equation}\label{eq:difference}
    \sigma(\widetilde{U}, \widetilde{w})-\sigma(U, \w)  
    = 2\sum_{i\in C} (c_i-1)w_i \sum_{j=1}^n w_j|\ip{\x_i}{\x_j}| + \sum_{i,j\in C}(c_i-1)(c_j-1)w_iw_j|\ip{\x_i}{\x_j}|.
\end{equation}
Now, 
\begin{align}\label{eq:linear_term}
  2\sum_{i\in C} (c_i-1)w_i \sum_{j=1}^n w_j|\ip{\x_i}{\x_j}| 
  & = 2\sum_{i\in C} (c_i-1)\lambda_1(B)w_i^2\nonumber\\
  & = 2\lambda_1(B)\sum_{i\in C} (c_i-1)w_i^2\nonumber\\
  & = 2\lambda_1(B)\left(\|\widetilde{\w}\|_2^2 - \|\w\|_2^2\right)\nonumber\\
  & = 0.  
\end{align}
By condition $(i)$, we have $\ip{\x_i}{\x_j} = |\ip{\x_i}{\x_j}|$ whenever $i,j\in C$, and so
\begin{equation}\label{eq:quadratic_term}
       \sum_{i,j\in C}(c_i-1)(c_j-1)w_iw_j|\ip{\x_i}{\x_j}|  = \left\|\sum_{i\in C}(c_i-1)w_i\x_i\right\|_2^2\ge 0.
\end{equation}

Combining \eqref{eq:difference}, \eqref{eq:linear_term} and \eqref{eq:quadratic_term}, we have
\[   \sigma(\widetilde{U}, \widetilde{w}) \ge \sigma(U, \w) = \lambda(r,n).\]
By Theorem \ref{thm:relative_top_r_sum_characterization}, we have
\[ \sigma(\widetilde{U}, \widetilde{w})\le \lambda(r,n).\]
We conclude that 
\[  \sigma(\widetilde{U}, \widetilde{w}) = \lambda(r,n),\]
implying
$(\widetilde{U}, \widetilde{\w})$ is a maximizing pair. 
\end{proof}

In the last claim, we will repeatedly apply cloning (Claim \ref{claim:cloning}) over different orthants of $\mathbb{R}^r$ to reduce the order of $U$. To find the vector $\cvec$ needed to apply Claim \ref{claim:cloning}, we will use Carath\'eodory's Theorem given below.

\begin{theorem}[Carath\'eodory's Theorem {\cite[Corollary 7.1i]{Schrijver_1986}}]\label{theorem:Caratheodory}
Let $S \subseteq \mathbb{R}^d$. If $\x\in\cone(S)$, then there exist points
$\s_1,\ldots,\s_d \in S$ and coefficients $a_1,a_1,\ldots,a_d \geq 0$ such that
\[
\x = \sum_{i=1}^d a_i \s_i.
\]
\end{theorem}

\begin{claim}[Stability]\label{claim:stability} We have 
\[\lambda(r,n)\le \lambda(r, N_r).\]
\end{claim}

\begin{proof} For $x\in \mathbb{R}$, define
\[ 
\sgn(x):=
\begin{cases}
   1 & \text{ if }x\ge 0;\\
   -1 & \text{ if }x<0.
\end{cases}
\]
For a vector $\x\in \mathbb{R}^r$, define 
\[\sgn(\x) := (\sgn(x_1), \ldots, \sgn(x_r))\in \{\pm 1\}^r.\]
For a vector $\s\in \{\pm 1 \}^r$, define 
\[O_{\s}:=\{\x\in \mathbb{R}^r: \sgn(\x) = \s\}.\] 
The set $O_{\s}$ is called the \emph{orthant} associated with $\s$. Clearly, the set of orthants $\{O_{\s}:\s \in \{\pm 1\}^r\}$ is a partition of $\mathbb{R}^r$.

Consider the vectors $\x_i\in \mathbb{R}^r$ $(i\in [n])$ defined previously. Fix an $\s\in \{\pm 1\}^r$, and let 
\[C_{\s} := \{i: \x_i\in O_{\s}\}\subseteq [n].\] 
Suppose that $C_{\s}\neq \emptyset$. We will first argue that $C_{\s}$ satisfies the conditions of Claim \ref{claim:cloning}. 

Observe that if $\x_i, \x_j\in O_{\s}$, then clearly $\ip{\x_i}{\x_j}\ge 0$; so condition $(i)$ holds. Now, the space $\mathbb{S}^{r\times r}$ of all real symmetric matrices of order $r$ has dimension ${r+1 \choose 2}$. The matrix 
\[ \frac{1}{|C_{\s}|}\sum_{i\in C_{\s}}\x_i\x_i^\top \]
lies in the convex hull of the set $\{\x_i\x_i^\top: i\in C_{\s}\}\subseteq \mathbb{S}^{r\times r}$. Thus, by Carath\'eodory's Theorem \ref{theorem:Caratheodory}, there exists a subset $I_{\s}\subseteq C_{\s}$ with $|I_{\s}|\le {r+1\choose 2}$ and numbers $a_i\ge 0$ $(i\in I_{\s})$ such that 
\begin{equation}\label{eq:cara_1}
\frac{1}{|C_{\s}|}\sum_{i\in C_{\s}}\x_i\x_i^\top  = \sum_{i\in I_{\s}}a_i\x_i\x_i^\top. 
\end{equation}

Consider the vector $\cvec\in \mathbb{R}^n$ given by 
\[ c_i = 
\begin{cases}
    a_i|C_{\s}|  & \text{ if }i\in I_{\s};\\
    0 & \text{ if }i\in C_{\s}\backslash I_{\s};\\
    1 & \text{ if }i\notin C_{\s}.
\end{cases}\]
Then, using \eqref{eq:cara_1},  
\[\sum_{i\in C_{\s}}c_i\x_i\x_i^\top = |C_{\s}|\sum_{i\in I_{\s}}a_i\x_i\x_i^\top = \sum_{i\in C_{\s}}\x_i\x_i^\top.\]
Thus, $C_{\s}$ and $\cvec$ satisfy the conditions of Claim \ref{claim:cloning}, and therefore one can find a maximizing pair $(\widetilde{U}, \widetilde{\w})$ for $\lambda(r,n)$ such that 
\[|J_{\widetilde{U}}\cap C_{\s}|\le |I_{\s}|\le {r+1\choose 2},\]
where $J_{\widetilde{U}}:=\{i: i\text{-th row of }\widetilde{U} \text{ is non-zero}\}$. 

Since $(\widetilde{U}, \widetilde{\w})$ is a maximizing pair, so all the previous (analogous) claims hold. Applying Claim \ref{claim:cloning} repeatedly for all $\s\in \{\pm 1\}^{r}$, we can find a maximizing pair $(\widetilde{U}, \widetilde{w})$ such that 
\begin{equation}\label{eq:w_tilde_estimate}
 |J_{\widetilde{U}}|\le 2^{r}{r+1\choose 2}= N_r. 
\end{equation}

Observe that $\widetilde w_i=0$ whenever $\widetilde\x_i=0$ because of the definitions of $\widetilde{w}_i$ and $\widetilde{\x}_i$. Deleting the rows of $\widetilde{U}$ corresponding to indices outside $J_{\widetilde{U}}$ gives a \(k\times r\) submatrix $M$ with $k := |J_{\widetilde{U}}|\le N_r$ that satisfies $M^\top M = I_r$, and hence $k\ge r$. Thus, we have shown that 
\begin{align*}
  \lambda(r,n)& \le \max\left\{\sum_{i,j=1}^{k} w_iw_j|\ip{\x_i}{\x_j}|: \, \w \in \mathbb{R}_{\ge 0}^{k},\, \|\w\|_2= 1,\, M\in \mathbb{R}^{k\times r},\, M^\top M = I_r\right\} \\
  & = \lambda(r, k)\\
  & \le \lambda(r, N_r),  
\end{align*}
and the claim holds.
\end{proof}

Taking the supremum over \(n\ge r\) in Claim~\ref{claim:stability}, we obtain
\[
\lambda(r)
=
\sup_{n\ge r}\lambda(r,n)
\le
\lambda(r,N_r)
\le
\lambda(r).
\]
Therefore,
\[
\lambda(r,N_r)=\lambda(r).
\]
Since \(\lambda(r,n)\) is nondecreasing in \(n\), it follows that
\[
\lambda(r,n)=\lambda(r)
\qquad\text{for every } n\ge N_r.
\]
This completes the proof of Theorem~\ref{thm:stability_rel_proj_constant}.

\section*{Acknowledgements}

Bojan Mohar is supported in part by the NSERC Discovery Grant R832714 (Canada), by the ERC Synergy grant (European Union, ERC, KARST, project number 101071836), and by the Research Project N1-0218 of ARIS (Slovenia). Seyed Ahmad Mojallal is partially supported by the ERC Synergy grant (European Union, ERC, KARST, project number 101071836). 

\section*{Declaration of AI use}

The authors acknowledge the use of ChatGPT (GPT-5.5, OpenAI; accessed May 2026) solely for preliminary brainstorming and the exploration of possible proof strategies. AI tools were not used to draft the manuscript. All formal statements, arguments, and proofs in the manuscript were written and checked by the authors, who take full responsibility for the accuracy and integrity of the article.

\bibliographystyle{plain}
\bibliography{finiteness_ref}

\vspace{0.4cm}
\affl{Hitesh Kumar}{hitesh.kumar.math@gmail.com, hitesh\_kumar@sfu.ca}{Department of Mathematics, Simon Fraser University, Burnaby, Canada}

\affl{Bojan Mohar}{mohar@sfu.ca}{Department of Mathematics, Simon Fraser University, Burnaby, Canada\\On leave from FMF, Department of Mathematics, University of Ljubljana.}

\affl{Seyed Ahmad Mojallal}{seyed\_ahmad\_mojallal@sfu.ca}{Department of Mathematics, Simon Fraser University, Burnaby, BC, Canada}

\affl{Shivaramakrishna Pragada}{shivaramkratos@gmail.com, shivaramakrishna\_pragada@sfu.ca}{Department of Mathematics, Simon Fraser University, Burnaby, Canada}

\end{document}